\documentclass[12pt,a4paper]{amsart}
\usepackage[T1]{fontenc}
\usepackage[utf8]{inputenc}
\usepackage{lmodern}
\usepackage{amsmath,amssymb,amsthm,amscd,mathtools}
\usepackage[margin=25mm,top=27mm,bottom=27mm]{geometry}
\usepackage{microtype}
\usepackage{tikz}
\usetikzlibrary{calc, arrows.meta}
\usepackage[hidelinks]{hyperref}
\hypersetup{pdftitle={On the orbits of the automorphism group of toroidal spherical varieties},pdfauthor={Alexander Chernov},pdfsubject={Automorphism orbits on affine and smooth complete toroidal spherical varieties}}
\newtheorem{theorem}{Theorem}
\newtheorem*{theorem*}{Theorem}
\newtheorem{proposition}{Proposition}
\newtheorem{lemma}{Lemma}
\newtheorem{corollary}{Corollary}
\theoremstyle{definition}
\newtheorem{definition}{Definition}
\newtheorem{example}{Example}
\newtheorem{remark}{Remark}
\newcommand{\K}{\mathbb K}
\newcommand{\Z}{\mathbb Z}
\newcommand{\Q}{\mathbb Q}
\newcommand{\A}{\mathbb A}
\newcommand{\PP}{\mathbb P}
\newcommand{\Ga}{\mathbb G_a}
\newcommand{\Gm}{\mathbb G_m}
\newcommand{\cO}{\mathcal O}
\newcommand{\cD}{\mathcal D}
\newcommand{\cF}{\mathcal F}
\newcommand{\cT}{\mathcal T}
\newcommand{\cS}{\mathcal S}
\newcommand{\Cl}{\operatorname{Cl}}
\newcommand{\Pic}{\operatorname{Pic}}
\newcommand{\Aut}{\operatorname{Aut}}
\newcommand{\Lie}{\operatorname{Lie}}
\newcommand{\Spec}{\operatorname{Spec}}
\newcommand{\supp}{\operatorname{Supp}}
\newcommand{\divisor}{\operatorname{div}}
\newcommand{\ord}{\operatorname{ord}}

\newcommand{\id}{\operatorname{id}}

\newcommand{\SL}{\operatorname{SL}}
\newcommand{\SO}{\operatorname{SO}}

\newcommand{\Hom}{\operatorname{Hom}}
\newcommand{\Nn}{\Z_{\geq 0}}
\newcommand{\gen}[1]{\langle #1\rangle}
\newcommand{\pair}[2]{\langle #1,#2\rangle}
\newcommand{\arxiv}[1]{\href{https://arxiv.org/abs/#1}{arXiv:#1}}
\numberwithin{equation}{section}
\title{On orbits of the automorphism group of toroidal spherical varieties}
\author{Alexander Chernov}
\address{\scshape Alexander Chernov \newline HSE University, Faculty of Computer Science, Pokrovsky Boulevard 11, Moscow, 109028 Russia \newline Lomonosov Moscow State University, Faculty of Mechanics and Mathematics, Department of Higher Algebra, Leninskie Gory 1, Moscow, 119991 Russia}
\email{chenov2004@gmail.com}
\author{Klim Rolgin}
\address{\scshape Klim Rolgin \newline Lomonosov Moscow State University, Faculty of Mechanics and Mathematics, Department of Higher Algebra, Leninskie Gory 1, Moscow, 119991 Russia}
\email{panamesaasem@gmail.com}
\thanks{This article is an output of a research project (HSE-BR-2025-22) implemented as part of the
Basic Research Program at HSE University.}

\date{}

\begin{document}
\begin{abstract}
In this paper we study the orbits of the automorphism group on a toroidal spherical variety. In the affine case, we prove that two points lie in the same orbit of the identity component if and only if the kernels of the homomorphisms from the divisor class group to their local divisor class groups coincide. These kernels are described in terms of $B$-stable prime divisors. For a smooth complete toroidal spherical variety, we obtain a similar criterion using the monoids generated by the classes of $B$-stable prime divisors that do not contain a given $G$-orbit.
\end{abstract}
\maketitle

\section*{Introduction}

Let $\K$ be an algebraically closed field of characteristic zero. By a variety or an algebraic group we always mean an algebraic variety or an algebraic group over~$\K$. By open and closed subsets of algebraic varieties we always mean open and closed subsets in the Zariski topology. We denote by~$\Ga$ the algebraic group~$(\K,+)$ and by~$\Gm$ the algebraic group~$(\K^\times,\cdot)$. Throughout the paper $G$ denotes a connected reductive algebraic group,~$B\subseteq G$ a fixed Borel subgroup, and $T\subseteq B$ a maximal torus.

\begin{definition}
    An irreducible normal $G$-variety $X$ is called \emph{spherical} if the Borel subgroup~$B$ acts on~$X$ with an open orbit.
\end{definition}

Any spherical variety $X$ has finitely many $B$-orbits and hence it has finitely many $G$-orbits; it was simultaneously proven by Brion in~\cite{Bri} and by Vinberg in~\cite{Vin86}. The naturally arising problem is to describe the orbits of the whole automorphism group $\Aut(X)$ or its identity component $\Aut^0(X)$. Obviously, points from the same $G$-orbit lie in the same $\Aut^0(X)$-orbit since the image of $G$ in $\Aut(X)$ lies in $\Aut^0(X)$, and hence we need to understand which $G$-orbits can be connected via some automorphism not lying in~$G$.

This problem was explicitly solved by Arzhantsev and Bazhov in the case of an affine toric variety in~\cite{AB}, and by Bazhov in the case of a complete toric variety in~\cite{Baz}. Toric varieties are a special case of spherical varieties; they arise when $G=B=T$. In the affine case Arzhantsev and Bazhov obtained a criterion for two $T$-orbits to lie in the same orbit of~$\Aut^0(X)$ in terms of the local divisor class group. In particular, the orbits of the identity component coincide with the orbits of the group generated by the acting torus and the torus-normalized one-parameter unipotent subgroups. For complete toric varieties, Bazhov described the $\Aut^0(X)$-orbits by using monoids generated by classes of invariant prime divisors that do not contain the $T$-orbit. Thus, both the affine and complete cases admit criteria expressed in terms of divisor classes, although the invariants are different.

The theory of spherical varieties provides a setting in which one can ask for analogs of these results. The theory of embeddings of homogeneous spaces was developed by Luna and Vust in~\cite{LV}. Their description of spherical embeddings by colored fans extends the combinatorial description of toric varieties; see~\cite{Kn91}. Timashev's monograph~\cite{Tim11} gives a systematic account of this theory and its applications.

The study of the orbits of the whole automorphism group is closely related to the concept of flexibility of affine varieties introduced in~\cite{AFKKZ}. An affine variety $X$ is called \emph{flexible} if the tangent vectors to $\Ga$-orbits span the tangent space at every smooth point of $X$. It is shown in~\cite{AFKKZ} that this is equivalent to the following condition: the subgroup of $\Aut^0(X)$ generated by all $\Ga$-subgroups acts transitively on the set of regular points of~$X$.

Any flexible variety has no non-constant invertible regular functions, but the converse is not true in general. First, Shafarevich proved that affine horospherical varieties of semisimple groups are flexible~\cite{Sh17}. Later, Gaifullin and Shafarevich proved in~\cite{GS19} that a normal affine horospherical variety is flexible if and only if it has no non-constant invertible regular functions; the non-normal case is treated in~\cite{GK26}. In the recent paper~\cite{Sh} Shafarevich finishes the study of the flexibility of affine spherical varieties, proving that $X$ is flexible if and only if any invertible regular function on $X$ is constant.

In this sense, the affine part of this paper continues Shafarevich's investigations. We pass from transitivity on the regular locus to a description of all $\Aut^0(X)$-orbits, including orbits which are contained in the singular locus of $X$.

The $B$-stable but not $G$-stable prime divisors on a spherical variety are called~\emph{colors}.

\begin{definition}
A $G$-orbit $Y$ in a spherical $G$-variety $X$ is called \emph{toroidal} if it is not contained in any color. The variety itself is called \emph{toroidal} if all its $G$-orbits are toroidal.
\end{definition}

In this paper, we prove orbit criteria for affine toroidal spherical varieties and for smooth complete toroidal spherical varieties. We obtain our results without studying the structure of the whole automorphism group of such varieties. This is a strongly related and important problem; in the horospherical case, interesting results were obtained in~\cite{BHK}.

First, we introduce the following invariants. For a~$G$-orbit $Y\subseteq X$, let $\cD^B(Y)$ be the set of all $B$-stable prime divisors not containing $Y$. Put
\begin{equation}\label{eq:invariants}
 S(Y)=\sum_{D\in\cD^B(Y)}\Z[D]\subseteq\Cl(X),
 \qquad
 \Gamma(Y)=\sum_{D\in\cD^B(Y)}\Nn[D]\subseteq\Cl(X).
\end{equation}

Our main results are the following two theorems.
\begin{theorem*}[Theorem~\ref{thm:affine}]
    Let $X$ be an affine toroidal spherical $G$-variety and $x,x'$ be two points on $X$. Then the following conditions are equivalent:
    \begin{enumerate}
    \item $x$ and $x'$ lie in the same $\Aut^0(X)$-orbit;
    \item $\Cl_x(X)=\Cl_{x'}(X)$;
    \item $S(Gx)=S(Gx')$.
    \end{enumerate}
\end{theorem*}

\begin{theorem*}[Theorem~\ref{thm:main-complete}]
    Let $X$ be a smooth complete toroidal spherical $G$-variety and $x, x'$ be two points on $X$. Then the following conditions are equivalent: \begin{enumerate}
        \item $x$ and $x'$ lie in the same $\Aut^0(X)$-orbit;
        \item $x$ and $x'$ lie in the same $\operatorname{GB}(X)$-orbit;
        \item $\Gamma(Gx)=\Gamma(Gx')$,
    \end{enumerate}
    where $\operatorname{GB}(X)$ denotes the subgroup of $\Aut^0(X)$ generated by $G$ and all $B$-root subgroups.
\end{theorem*}

Theorem~\ref{thm:affine} gives a partial answer to~\cite[Problem~2]{Sh}, which asks for a description of the $\Aut(X)$-orbits on affine spherical varieties. It also answers the question raised in~\cite[Remark~5.4]{AB} for affine toroidal spherical varieties: for such varieties, conditions~1 and~3 of~\cite[Theorem~5.1]{AB} are equivalent, i.e. two points $x,x'\in X$ lie in the same $\Aut^0(X)$-orbit if and only if $\Cl_x(X)=\Cl_{x'}(X)$. For horospherical varieties, a necessary condition of a similar type in terms of the divisor class group was obtained in~\cite{Gai}, and the $\Aut(X)$-orbits were described in~\cite{BGS} in terms of degrees of homogeneous locally nilpotent derivations. In~\cite{AZ25} it is shown that in a quasiaffine spherical variety every $G$-orbit contained in the regular locus can be connected by a $B$-root subgroup with any minimal $G$-orbit containing it in its closure; see also~\cite{AZ23}. 

The text consists of three parts. In Section~\ref{sec:prelim}, we fix the notation and recall some basic facts about spherical varieties. In Section~\ref{sec:affine}, we prove the necessary facts about local divisor class groups and establish the affine criterion. In Section~\ref{sec:complete}, we construct $B$-root subgroups that realize transitions between $G$-orbits and deduce from this the criterion in the complete case.

\subsection*{Acknowledgments}
The authors are grateful to Anton Shafarevich for his help with the material of this paper and to Dmitry Timashev for the excellent advanced course on spherical varieties that enabled the authors to begin their study of this theory.
The first author is a scholarship holder of the Foundation for the Development of Theoretical Physics and Mathematics ``BASIS''.

\section{Basic facts about spherical varieties}\label{sec:prelim}

In this section we recall some basic facts about spherical varieties. All the following statements can be found in~\cite[Chapter 3]{Tim11}.

We fix a connected reductive group $G$, a Borel subgroup $B\subseteq G$, and a maximal torus $T\subseteq B$. Write $\mathfrak X(T)$ for the character lattice of $T$, identified with $\mathfrak X(B)$.

Let $X$ be a spherical $G$-variety. The weight lattice of $X$ is
\[
 M=\{m\in\mathfrak X(B)\mid \K(X)^{(B)}_m\ne 0\},
 \qquad N=\Hom(M,\Z),
\]
where $\K(X)^{(B)}_m$ is the space of rational $B$-semiinvariants of weight $m$. Every such space is one-dimensional if it is nonzero, since $B$ has an open orbit. For each $m\in M$, choose a nonzero element $f_m\in\K(X)^{(B)}_m$. Put $M_\Q=M\otimes_\Z\Q$ and $N_\Q=N\otimes_\Z\Q$.

Denote by $\cD^B$ the set of $B$-stable prime divisors in $X$ and by $\cD^G$ its subset of $G$-stable prime divisors. The $B$-stable but not $G$-stable prime divisors, i.e. the elements of the set
\[
 \cD=\cD^B\setminus\cD^G,
\]
are called \emph{colors}.

For $D\in\cD^B$, define $v_D\in N$ by the following formula:
\[
 \pair{v_D}{m}=\ord_D(f_m).
\]
As we mentioned above, $\K(X)_m^{(B)}$ is one-dimensional and therefore this formula does not depend on the choice of $f_m$. Any $G$-invariant $\Q$-valuation~$v$ of $\K(X)$ defines the element $\overline{v} \in N_\Q$ by restricting $v$ to $M$ and extending to $M_\Q$. The set of all such $\overline{v}$ forms the valuation cone $\mathcal V\subseteq N_\Q$. One of the main facts of this theory is that any $G$-invariant valuation $v$ is uniquely determined by $\overline{v}$.

For a $G$-orbit $Y$, let
\[
 X_Y=\{x\in X\mid Y\subseteq\overline{Gx}\},
 \qquad
 X_Y^0=X_Y\setminus\bigcup_{\substack{F\in\cD\\Y\not\subseteq F}}F.
\]
The variety $X_Y$ is an open spherical subvariety with a unique closed orbit $Y$. A spherical variety with a unique closed orbit is called~\emph{simple}. The set $X_Y^0$ is affine and is called the \emph{open $B$-chart} of $X_Y$ or of the orbit~$Y$.

Fix the orbit $Y$ and put $P = \{g\in G  \ | \ gX_Y^0 = X_Y^0\}$. Since $P$ contains $B$, it is a parabolic subgroup and therefore $P=P_u \rtimes L$, where $P_u$ is the unipotent radical of~$P$ and~$L$ is the Levi subgroup of~$P$ containing~$T$. The following theorem is one of the central theorems in this theory.

\begin{theorem}[{The Local Structure Theorem~\cite{BLV}, \cite[Section~1.5]{BL}, \cite[Theorem~15.17]{Tim11}}]\label{LST}
    In the notation above there exists a closed $L$-invariant subvariety $Z_Y \subset X_Y^0$ such that: \begin{enumerate}
        \item $X_Y^0 = P\cdot Z_Y \simeq P\times^L Z_Y \simeq P_u \times Z_Y$;
        \item $Z_Y$ is an affine $L$-spherical variety.
    \end{enumerate}
    Moreover, there exists a natural bijection between $G$-orbits in $X$ intersecting $X_Y^0$ and $L$-orbits in $Z_Y$, established as follows: $Y'\mapsto Y'\cap Z_Y$.
\end{theorem}

When $Y$ is toroidal, it turns out that $[L,L]$ acts trivially on $Z_Y$ and hence $Z_Y$ is an affine toric variety with respect to the quotient of the torus $T'=L/[L,L]$ by the kernel of its action. The $G$-orbits meeting $X_Y^0$ correspond to the $T'$-orbits in $Z_Y$, and
\[
 Y\cap X_Y^0=P_u\times O_Y,
\]
where $O_Y$ is the unique closed $T'$-orbit in $Z_Y$.

The following lemma shows that in the case of an affine spherical variety, the open $B$-chart of an orbit $Y$ is not only an open affine subset but also principal.

\begin{lemma}\cite[Proposition 1]{Sh}\label{lem:principal-chart}
Let $X$ be an affine spherical $G$-variety and let $Y\subseteq X$ be a $G$-orbit. There exists a nonzero regular $B$-semiinvariant $f_Y\in\K[X]$ such that
\[
 X_Y^0=X_{f_Y} = X\setminus \{f_Y=0\}.
\]
\end{lemma}

This property allows us to pass from $\Ga$-actions on $X_Y^0$ to actions on $X$, using Lemma~\ref{lem:replica} below.

\section{The affine case}\label{sec:affine}

\subsection{Local divisor class groups} Here we introduce the local divisor class group and prove the statements about it needed for the affine criterion.

\begin{definition}
    Let $X$ be a normal variety and $\operatorname{WDiv}(X)$ be the group of Weil divisors on~$X$. For $x\in X$ the group \begin{equation*}
        \Cl(X,x)=\operatorname{WDiv}(X)/\operatorname{LPDiv}(X,x),
    \end{equation*}
    where $\operatorname{LPDiv}(X,x)$ is the group of Weil divisors that are principal in some neighborhood of $x$, is called~\emph{the local divisor class group at the point $x$}.
\end{definition}

 Denote the kernel of the natural projection $\Cl(X)\to\Cl(X,x)$ by $\Cl_x(X)$. A class $[D]$ belongs to $\Cl_x(X)$ if and only if there is a divisor $D'\sim D$ whose support does not contain $x$. Indeed, if $D=\divisor(f)$ near $x$, then $D'=D-\divisor(f)$ is such a representative and vice versa. This observation makes the following proposition obvious.

\begin{proposition}\label{prop:locact}
    Let $X$ be a normal variety and $\varphi \in \Aut(X)$. Then for any $x\in X$ the following equality holds: \begin{equation*}
        \Cl_{\varphi(x)}(X) = \overline{\varphi}(\Cl_x(X)),
    \end{equation*}
    where $\overline{\varphi}$ denotes the image of $\varphi$ in $\Aut(\Cl(X))$, i.e. $\overline{\varphi}([D])=[\varphi(D)]$.
\end{proposition}

The next obvious lemma shows that the introduced group is indeed local, i.e. it does not depend on the choice of the neighborhood of the point.

\begin{lemma}\label{lem:open}
Let $X$ be a normal variety, let $U\subseteq X$ be an open subset, and let $x\in U$. There is a commutative diagram
\[
\begin{CD}
\Cl(X) @>>> \Cl(X,x)\\
@V{\pi_U}VV @VV{\pi_{U,x}}V\\
\Cl(U) @>>> \Cl(U,x),
\end{CD}
\]
where the vertical arrows send the class of $D$ to the class of $D|_U$. All the arrows are surjective, and $\pi_{U,x}$ is an isomorphism. In particular,
\[
 \Cl_x(X)=\pi_U^{-1}\bigl(\Cl_x(U)\bigr).
\]
\end{lemma}

\begin{proof} The horizontal arrows are surjective by definition. Now let us check that $\pi_{U,x}$ is well-defined and is an isomorphism, the rest is obvious.

\emph{Well-definedness.} If $D=\divisor(f)$ on a neighborhood $W$ of $x$ in $X$, then $D|_U=\divisor(f)$ on~$U\cap W$.

\emph{Injectivity.} Suppose that $[D|_U]=0$ in $\Cl(U,x)$. Then $D|_U=\divisor(f)$ in some neighborhood $W\subseteq U$ of $x$, where $f$ is a nonzero rational function on $X$. The divisor $D-\divisor(f)$ on $X$ has support disjoint from $W$, which yields $[D]=0$ in $\Cl(X,x)$.

\emph{Surjectivity.} This follows from the surjectivity of $\pi_U$.
\end{proof}

\begin{corollary}\label{cor:open-kernels}
In the notation of Lemma~\ref{lem:open}, let $D_1,\ldots,D_k$ be the prime divisors lying in the complement of $U$. Then there is an exact sequence
\[
 \bigoplus_{i=1}^k\Z[D_i]\longrightarrow\Cl_x(X)
 \xrightarrow{\pi_U}\Cl_x(U)\longrightarrow 0.
\]
\end{corollary}

\begin{proof}
It is well known that $\ker\pi_U=\sum_{i=1}^k\Z[D_i]$. The statement immediately follows from this fact and Lemma~\ref{lem:open}.
\end{proof}

The following lemma shows that the affine space factor does not affect the local divisor class group as it does not affect the usual class group. It is important when we use the Local Structure Theorem, which yields the unipotent group as a factor.

\begin{lemma}\label{lem:product}
Let $X$ be a normal variety and let $\pi_X\colon\A^n\times X\to X$ be the natural projection. For $p\in\A^n$ and $x\in X$, there is a commutative diagram
\[
\begin{CD}
\Cl(X) @>{\pi_X^*}>> \Cl(\A^n\times X)\\
@VVV @VVV\\
\Cl(X,x) @>{\pi_X^*}>> \Cl(\A^n\times X,(p,x)),
\end{CD}
\]
in which both horizontal arrows are isomorphisms. In particular,
\[
 \pi_X^*\bigl(\Cl_x(X)\bigr)=\Cl_{(p,x)}(\A^n\times X).
\]
\end{lemma}

\begin{proof}
It is sufficient to prove the statement for $n=1$.

It is well known that the upper arrow is an isomorphism. We now check the lower arrow.

\emph{Well-definedness.} If $D$ is principal on a neighborhood $U$ of $x$, then $\pi_X^*D$ is principal on~$\A^1\times U$.

\emph{Surjectivity.} This follows from the surjectivity of the upper arrow and the two vertical arrows.

\emph{Injectivity.} Suppose that $\pi_X^*[D]=0$ in $\Cl(\A^1\times X,(p,x))$. Then, for a neighborhood $W$ of $(p,x)$ and a nonzero rational function $f\in\K(\A^1\times X)$, we have
\[
 \pi_X^*D=\divisor(f)+E,
 \qquad \supp E\cap W=\varnothing.
\]

Consider the fiber $X_p=\{p\}\times X$. Neither $\pi_X^*D$ nor $E$ contains $X_p$ as a prime component. For $E$ it follows from $W \cap X_p \neq \varnothing$. Thus $f$ has no zeros or poles along $X_p$ and hence~$f|_{X_p} \in \K(X)^\times$. Restricting everything to~$X_p$, we obtain $D = \divisor(f|_{X_p}) + E|_{X_p}$, which yields $x \notin \supp(E|_{X_p})$. Therefore, $D$ is principal in some neighborhood of $x$ and hence~$[D]=0$ in $\Cl(X, x)$.
\end{proof}

\subsection{Identity component of the automorphism group.}\label{sec:aut0} We need to specify the convention for $\Aut^0(X)$ in the case of affine $X$ where it might not admit the structure of an algebraic group.

Let $\Aut^0(X)$ denote the set of automorphisms $\varphi \in \Aut(X)$ such that there exist an irreducible rational curve $C$, a morphism $\psi\colon C\times X \to X$, and points $c_1,c_2\in C$ such that $\psi_c = \psi(c, \cdot)$ is an automorphism of $X$ for any $c\in C$, $\psi_{c_1} = \id_X$ and $\psi_{c_2} = \varphi$. This convention is also used in~\cite[Definition~1.2]{AB}; it is easy to check that $\Aut^0(X)$ is a subgroup of $\Aut(X)$; see~\cite[Section~1]{AB}.

Note that if a connected linear algebraic group acts on $X$, then its image in $\Aut(X)$ lies in $\Aut^0(X)$. Indeed, every element of such a group lies in some Borel subgroup, and a Borel subgroup is isomorphic as a variety to $(\K^\times)^r\times\K^m$, so any two of its points can be connected by a rational curve.

For smooth complete spherical varieties, $\Aut^0(X)$ coincides with the usual identity component $\Aut(X)^0$ of the automorphism group. Indeed, $X$ is rational since it contains an open $B$-orbit, so $\Aut(X)^0$ is a connected linear algebraic group. Hence we have $\Aut(X)^0\subseteq\Aut^0(X)$. Conversely, for $\psi\colon C\times X\to X$ as above, $c\mapsto\psi_c$ is a morphism from $C$ to $\Aut(X)$ whose image is irreducible and contains $\id_X$, and hence $\psi_{c}\in\Aut(X)^0$.

\begin{lemma}\label{lem:trivial}
For a normal variety $X$, the group $\Aut^0(X)$ acts trivially on $\Cl(X)$.
\end{lemma}

\begin{proof}
    Let $\varphi$ be an automorphism lying in $\Aut^0(X)$. Consider the rational curve $C$ giving a family of automorphisms containing $\varphi$ and $\id_X$. Replacing $C$ with its normalization and removing a finite number of points we can assume that $C$ is an open subset of $\A^1$.

    Now we have the mapping $\psi:C\times X\to X$, $C = \A^1\setminus\{a_1,\dotsc,a_s\}$, $\psi_{c_1} = \id_X$, $\psi_{c_2}=\varphi$. Consider the projection $\pi_X:C\times X\to X$ and the isomorphism $\Phi:C\times X \to C\times X$ defined as follows: \begin{equation*}
        \Phi(c,x) = (c, \psi_c(x)).
    \end{equation*}

    Now take any divisor $D$ on $X$ and put $\widetilde D = \pi_X^*D$, $\overline{D} = \Phi^*(\widetilde D)$. Since $\Cl(C\times X) \simeq \Cl(X)$, we obtain $\overline{D} \sim \pi_X^* E$ for some divisor $E$ on $X$, i.e. \begin{equation*}
        \overline{D} = \pi_X^*E + \divisor(f(c,x)), \qquad f\in\K(C\times X)^\times.
    \end{equation*}

    For any $c\in C$, neither $\overline{D}$ nor $\pi_X^*E$ has $\{c\}\times X$ as a prime component. Hence $f$ has neither zeros nor poles along $\{c\}\times X$, and $f|_{\{c\}\times X}\in\K(X)^\times$. This means that we can restrict this equality to the fibers with $c = c_1$ and $c=c_2$, which yields~$D\sim E$ and $\varphi^*D \sim E$. Thus~$\varphi^*D\sim D$ for every divisor~$D$. Applying this to~$\varphi(D)$ gives~$\varphi(D)\sim D$. This finishes the proof.
\end{proof}

\begin{corollary}\label{cor:locclconst}
   Let $X$ be a normal variety. If $x$ and $y$ lie in the same orbit of $\Aut^0(X)$, then $\Cl_x(X)=\Cl_y(X)$.
\end{corollary}

\begin{proof}
    This follows immediately from Proposition~\ref{prop:locact} and Lemma~\ref{lem:trivial}.
\end{proof}

\subsection{The orbit criterion}

First of all, we recall the results obtained in~\cite{AB}.

For a toric variety $Z$ with the acting torus $\mathbb T$ and a $\mathbb T$-orbit~$O$ put \begin{equation*}
    S_{\mathbb T}(O) = \sum_{\substack{D: \ O \not\subseteq D \\ \text{$D$ is $\mathbb T$-stable}}}\Z[D]  \ \subseteq \ \Cl(Z).
\end{equation*}Let $\operatorname{AT}(Z)$ be the subgroup of $\Aut(Z)$ generated by $\mathbb T$ and the $\mathbb T$-normalized one-parameter unipotent subgroups (also called \emph{root subgroups}). It is obvious from the definition that $\operatorname{AT}(Z) \subseteq \Aut^0(Z)$. The following result describes the orbits of the identity component of the automorphism group of a toric variety.

\begin{theorem}\cite[Theorem~5.1 and Proposition~2.1]{AB}\label{thm:toric}
Let $Z$ be an affine toric variety and let~$z,z'\in Z$. The following conditions are equivalent:
\begin{enumerate}
\item $z$ and $z'$ lie in the same $\Aut^0(Z)$-orbit;
\item $z$ and $z'$ lie in the same $\operatorname{AT}(Z)$-orbit;
\item $\Cl_z(Z)=\Cl_{z'}(Z)$;
\item $S_{\mathbb T}(\mathbb Tz)=S_{\mathbb T}(\mathbb Tz')$.
\end{enumerate}
Moreover, $\Cl_z(Z)=S_{\mathbb T}(\mathbb Tz)$ for every $z\in Z$.
\end{theorem}

Now we show that this result has a direct analog in the case of affine toroidal spherical varieties.

Recall the subgroup $S(Y)\subseteq\Cl(X)$ defined in~\eqref{eq:invariants}. The following proposition holds for any toroidal orbit in any spherical variety.

\begin{proposition}\label{prop:local-spherical}
Let $Y$ be a toroidal $G$-orbit in a spherical variety $X$. Then
\[
 \Cl_y(X)=S(Y)\qquad\text{for every }y\in Y.
\]
\end{proposition}

\begin{proof}
We use the notation of Section~\ref{sec:prelim}. By Corollary~\ref{cor:locclconst}, the kernel $\Cl_y(X)$ is constant along $Y$, since the image of $G$ in $\Aut(X)$ lies in $\Aut^0(X)$. We may therefore choose $y\in Y\cap X_Y^0$ and use Theorem~\ref{LST}. Write $y=(u,z)$, where $z$ belongs to the unique closed torus orbit of $Z_Y$ which we denoted by~$O_Y$. Every torus-invariant prime divisor on $Z_Y$ contains that orbit and hence $S_{T'}(O_Y) = 0$. Theorem~\ref{thm:toric} gives $\Cl_z(Z_Y)=0$, and Lemma~\ref{lem:product} gives $\Cl_y(X_Y^0)=0$ since $X_Y^0 \simeq P_u \times Z_Y$ and $P_u$ is isomorphic to affine space as an algebraic variety.

The prime divisors lying in $X\setminus X_Y^0$ are precisely the $B$-stable prime divisors not containing $Y$. Indeed, the prime divisors contained in $X\setminus X_Y$ are exactly the $G$-stable prime divisors not containing $Y$, and all the colors are removed on passing to $X_Y^0$ since $Y$ is toroidal. Putting $U = X_Y^0$ in Corollary~\ref{cor:open-kernels} we obtain $\Cl_y(X) = S(Y)$.
\end{proof}

The following lemma is well known. Nevertheless, the proof is quite simple so we present it here. 

\begin{lemma}\label{lem:replica}
Let $X$ be an irreducible affine variety and $0\ne f\in\K[X]$. Then for every $\Ga$-subgroup $H\subseteq\Aut(X_f)$ there is a $\Ga$-subgroup $H'\subseteq\Aut(X)$ such that $H'x=Hx$ for all $x\in X_f$.
\end{lemma}

\begin{proof}
Let $\delta$ be the locally nilpotent derivation of $\K[X_f]=\K[X]_f$ corresponding to $H$. We have $\delta(f)=0$ since $f$ is invertible on $X_f$. Therefore, for big enough $N$, $f^N\delta$ is a locally nilpotent derivation of $\K[X]$ (this is called \emph{a replica of $\delta$}; see~\cite{Fr}). The corresponding $\Ga$-subgroup $H'\subseteq\Aut(X)$ has the same orbits on $X_f$ as $H$.
\end{proof}

Now we are ready to prove the main result of this section.

\begin{theorem}\label{thm:affine}
Let $X$ be an affine toroidal spherical $G$-variety and let $x,x'\in X$. The following conditions are equivalent:
\begin{enumerate}
\item $x$ and $x'$ lie in the same $\Aut^0(X)$-orbit;
\item $\Cl_x(X)=\Cl_{x'}(X)$;
\item $S(Gx)=S(Gx')$.
\end{enumerate}
\end{theorem}

\begin{proof}
The implication $(1)\Rightarrow(2)$ follows from Corollary~\ref{cor:locclconst} and the equivalence $(2)\Leftrightarrow(3)$ follows from Proposition~\ref{prop:local-spherical}.

$(2)\Rightarrow(1)$. We remind the reader, that there is only one closed $G$-orbit $Y$ on $X$ because the categorical quotient $X/\!\!/G$ is one point. Consider the principal open~$B$-chart $X_Y^0=X_{f_Y}\simeq P_u\times Z_Y$. Every $G$-orbit meets this chart since $X$ is toroidal. Hence acting by elements of $G$ we may assume that $x$ and $x'$ lie in $X_Y^0$. Then, acting by elements of $P_u$, we may assume
\[
 x=(e,z),\qquad x'=(e,z')
\]
for some $z,z'\in Z_Y$. These operations preserve condition (2). Lemmas~\ref{lem:open} and~\ref{lem:product} yield the equality $\Cl_z(Z_Y)=\Cl_{z'}(Z_Y)$. By Theorem~\ref{thm:toric}, there is an element $\varphi \in\operatorname{AT}(Z_Y)$ sending $z$ to $z'$.

Write $\varphi=g_k\circ\dots\circ g_1$, where each $g_i$ lies either in the acting torus of $Z_Y$ or in a root subgroup $H_i\subseteq\Aut(Z_Y)$. We can do this by $(1)\Leftrightarrow(2)$ in Theorem~\ref{thm:toric}. Put~$z_0=z$ and~$z_i=g_i(z_{i-1})$, so that $z_k=z'$. We claim that $(e,z_{i-1})$ and $(e,z_i)$ lie in the same~$\Aut^0(X)$-orbit for every $i$. If $g_i$ lies in the torus, then it lifts to an element of the Levi subgroup $L$, since $L$ maps onto this torus, and this element sends $(e,z_{i-1})$ to $(e,z_i)$. If~$g_i\in H_i$, then~$H_i$ acts on $X_Y^0 = P_u\times Z_Y$ trivially on the first factor. Since $X_Y^0=X_{f_Y}$ is a principal open subset, Lemma~\ref{lem:replica} gives a $\Ga$-subgroup $H_i'\subseteq\Aut^0(X)$ with the same orbits on $X_Y^0$ as $H_i$. In particular, $(e,z_{i-1})$ and $(e,z_i)$ lie in the same $H_i'$-orbit. Therefore~$x$ and~$x'$ lie in the same $\Aut^0(X)$-orbit. This finishes the proof.
\end{proof}

\begin{remark}
    Note that if $x$ is a smooth point on $X$, then $\Cl_x(X)=\Cl(X)$ since any divisor is principal in some neighborhood of $x$. Theorem~\ref{thm:affine} therefore says that all smooth points lie in the same $\Aut^0(X)$-orbit. This is consistent with Shafarevich's result on the flexibility of affine spherical varieties.
\end{remark}

Recall that $\overline{\varphi}$ denotes the action of $\varphi\in\Aut(X)$ on $\Cl(X)$; see Proposition~\ref{prop:locact}. 

\begin{corollary}\label{cor:full-affine}
Let $X$ be an affine toroidal spherical variety. Two $G$-orbits $Y,Y'$ are contained in the same $\Aut(X)$-orbit if and only if there is an automorphism $\varphi\in\Aut(X)$ such that
\[
 \overline{\varphi}(S(Y))=S(Y').
\]
\end{corollary}

\begin{proof}
Immediately follows from Proposition~\ref{prop:locact} and Theorem~\ref{thm:affine}.
\end{proof}

This corollary reduces the description of the $\Aut(X)$-orbits to the study of the image of~$\Aut(X)$ in~$\Aut(\Cl(X))$.

The following example illustrates Theorem~\ref{thm:affine} on an affine toroidal spherical variety with several $G$-orbits that does not admit the structure of a toric variety. Note that the homogeneous space $Q=\SL_2/T_0$ from this example is itself an affine toroidal spherical variety that is not toric.

\begin{example}\label{ex:affine}
Consider
\[
 Q=(\PP^1\times\PP^1)\setminus\Delta\simeq\SL_2/T_0,
 \qquad
 V=\Spec\K[u,v,w]/(uw-v^2),
\]
where $\Delta$ is the diagonal in $\PP^1\times\PP^1$ and $T_0\simeq\Gm\subseteq\SL_2$ is a maximal torus. We note that $Q$ is affine since, for example, $\Delta$ is an ample divisor or by Matsushima's criterion from the representation $Q=\SL_2/T_0$. The variety $V\simeq\A^2/(\Z/2\Z)$ is the standard quadratic cone and it is normal. Denote its acting torus by $\mathbb T$ and put
\[
 X=Q\times V,\qquad G=\SL_2\times\mathbb T.
\]
The variety $X$ is $G$-spherical but it is not toric even for a different torus action. It follows from the fact that the singular locus of $X$ is $Q\times\{0\}\simeq Q$. If $X$ were toric with respect to some torus $\widetilde T$, then $Q$ would be a normal toric variety. But $Q$ is not toric since it is smooth and $\Cl(Q) \simeq \Z$. Now we remind the reader that every smooth affine toric variety is isomorphic to $\A^k\times\Gm^l$ and hence has trivial divisor class group.

Now note that $X$ is toroidal since any color has the form $D\times V$, where $D$ is a divisor on~$Q$, and all $G$-orbits have the form $Q\times Y$, where $Y$ is a $\mathbb T$-orbit on $V$. One can easily obtain \begin{equation*}
    \Cl(X) = \Z[H]\oplus (\Z/2\Z)[E],
\end{equation*} where $H = (\A^1\times \{\infty\})\times V$ and $E = Q \times \{u=v=0\}$. We also have
\[
 S(Q\times\{0\})=\Z[H],
 \qquad
 S(Q\times Y)=\Z[H]\oplus(\Z/2\Z)[E]
 \quad(Y\ne\{0\}).
\]
Hence Theorem~\ref{thm:affine} gives exactly two $\Aut^0(X)$-orbits:
\[
 O_1 = Q\times\{0\},\qquad O_2 = Q\times(V\setminus\{0\}).
\]

Note that they are also two $\Aut(X)$-orbits since $O_1$ is the singular locus of~$X$.
\end{example}

\section{The complete case}\label{sec:complete}

In this section, $X$ is a smooth complete toroidal spherical $G$-variety. In this case $\Aut^0(X)$ is a connected linear algebraic group (see Section~\ref{sec:aut0}), $\Cl(X)=\Pic(X)$ and it is generated by the classes of $B$-stable prime divisors. We keep the notation of Section~\ref{sec:prelim}.


\subsection{The fan and the orbit monoids}\label{subsec:monoids} In the complete case we need the combinatorial description of spherical varieties. One can find more information on it in~\cite[Chapter 3]{Tim11}.

Spherical embeddings of the fixed spherical homogeneous space $G/H$, forming the open~$G$-orbit, are classified by colored fans in $N_\Q$. If $X$ is complete, then its fan $\cF$ covers the valuation cone $\mathcal V$. 

The colored fan of a toroidal spherical variety has no colors attached to its cones, and the support of the fan is contained in $\mathcal V$. If $X$ is toroidal and complete, then the support of its fan $\cF$ is exactly $\mathcal V$.

Each cone $\sigma\in\cF$ corresponds to a $G$-orbit $Y_\sigma$ and
\begin{equation}\label{eq:closure}
 Y_\sigma\subseteq\overline{Y_\tau}
 \quad\Longleftrightarrow\quad \tau\text{ is a face of }\sigma.
\end{equation}
The rays $\rho\in\cF(1)$ correspond to the $G$-stable prime divisors $D_\rho$. Let $v_\rho\in N$ be the primitive vector of $\rho$. Then $v_\rho=\overline{v_{D_\rho}}$ (see Section~\ref{sec:prelim}) and
\begin{equation}\label{eq:incidence}
 Y_\sigma\subseteq D_\rho
 \quad\Longleftrightarrow\quad\rho\in\sigma(1).
\end{equation}
Smoothness of $X$ implies that each cone is \emph{regular}, i.e. it is generated by part of a basis of~$N$. In particular, every subset of $\sigma(1)$ spans a face of $\sigma$. Note also that, by~\eqref{eq:incidence}, a~$G$-orbit is uniquely determined by the set of $G$-stable prime divisors containing it. The latest obviously holds for any toroidal spherical variety since any cone is uniquely determined by the set of its rays.

Recall the monoids $\Gamma(Y_\sigma)\subseteq\Cl(X)=\Pic(X)$ defined in~\eqref{eq:invariants}. Since $X$ is toroidal, no color contains $Y_\sigma$, and we can rewrite
\begin{equation}\label{eq:monoid-fan}
 \Gamma(Y_\sigma)=
 \sum_{F\in\cD}\Nn[F]
 +\sum_{\rho\in\cF(1)\setminus\sigma(1)}\Nn[D_\rho].
\end{equation}

\begin{remark}
    The use of the monoid instead of the subgroup is crucial. Since $X$ is smooth, every divisor is locally principal and therefore $\Cl_x(X)=\Cl(X)$ for all $x\in X$. By Proposition~\ref{prop:local-spherical} we get $S(Y)=\Cl_x(X)=\Pic(X)$ for every $G$-orbit $Y$.
\end{remark}

The proposition below shows that the monoid $\Gamma(Y)$ admits a pure geometric description. This is an analog of Proposition~\ref{prop:local-spherical} in the smooth case.

\begin{proposition}\label{prop:intrinsic}
For a $G$-orbit $Y\subseteq X$ and a point $y\in Y$ the following equality holds:
\begin{equation}\label{eq:intrinsic}
 \Gamma(Y)=\Gamma(y), \quad \text{where}\quad \Gamma(y)=\{[\mathcal L]\in\Pic(X)\mid
 \exists s\in H^0(X,\mathcal L): \ s(y)\ne0\}.
\end{equation}
\end{proposition}

\begin{proof}
The set $\Gamma(y)$ is a monoid since if two sections $s\in H^0(X, \mathcal L)$ and $s' \in H^0(X, \mathcal L')$ are nonzero at $y$, then $s\otimes s' \in H^0(X, \mathcal L\otimes\mathcal L')$ is also nonzero at $y$.

Let $D$ be a $B$-stable prime divisor not containing $Y$. Then there exists $g\in G$ such that $gy\notin D$. By Lemma~\ref{lem:trivial} we have $[D]=[g^{-1}D]$. The canonical section of $\cO_X(g^{-1}D)$ is nonzero at $y$ and hence $[D]\in\Gamma(y)$. This proves the inclusion $\Gamma(Y)\subseteq\Gamma(y)$.

Conversely, take a line bundle $\mathcal L$ with a section $s$ such that $s(y)\ne 0$. We may assume that $\mathcal L$ admits a $G$-linearization by replacing $G$ with its finite cover. This does not change the orbits. The restriction homomorphism
\[
 r\colon H^0(X,\mathcal L)\longrightarrow H^0(Y,\mathcal L|_Y)
\]
is a nonzero homomorphism of $G$-modules because $r(s)$ is nonzero. Since $H^0(X,\mathcal L)$ is finite-dimensional ($X$ is complete) and $G$ is reductive, it has a simple submodule $V$ such that the restriction $r|_V$ is injective. Let $s_0\in V$ be a nonzero highest weight vector. Then~$r(s_0)\ne0$ and hence the support of the effective $B$-stable divisor $E=\divisor(s_0)$ does not contain $Y$. Consequently, every prime component of $E$ belongs to $\cD^B(Y)$ and
\[
 [\mathcal L]=[E]\in\Gamma(Y).
\]
This proves the reverse inclusion.
\end{proof}

\begin{corollary}\label{cor:necessary}
If two $G$-orbits $Y$ and $Y'$ lie in the same $\Aut^0(X)$-orbit, then $\Gamma(Y)=\Gamma(Y')$.
\end{corollary}

\begin{proof}
The group $\Aut^0(X)$ acts trivially on $\Pic(X)$ and for any $\varphi \in \Aut(X)$ we have~$\Gamma(\varphi(y)) = \overline{\varphi}(\Gamma(y))$, where $\overline{\varphi}$ denotes the image of $\varphi$ in $\Aut(\Pic(X))$. Let~$y \in Y$,~$y'\in Y'$ and $\psi\in\Aut^0(X)$ such that $\psi(y)=y'$. We have \[
\Gamma(Y')=\Gamma(y')=\Gamma(\psi(y))=\overline{\psi}(\Gamma(y))=\Gamma(y)=\Gamma(Y).
\]
This finishes the proof.
\end{proof}

The converse of Corollary~\ref{cor:necessary} is proved by moving from one $G$-orbit to another step by step. At every step we remove one ray from a cone. Namely, let $\sigma\in\cF$ and $\rho\in\sigma(1)$, and let $\sigma'$ be the face of $\sigma$ generated by $\sigma(1)\setminus\{\rho\}$. It exists since $\sigma$ is regular. By~\eqref{eq:closure} and~\eqref{eq:incidence}, the orbit $Y_\sigma$ lies in the closure of $Y_{\sigma'}$, and $Y_{\sigma'}$ is not contained in $D_\rho$. We will show that $Y_\sigma$ and $Y_{\sigma'}$ lie in the same $\Aut^0(X)$-orbit if and only if
\begin{equation}\label{eq:deletion}
 [D_\rho]\in\Gamma(Y_\sigma).
\end{equation}
The necessity is obvious. Indeed, the divisor $D_\rho$ does not contain $Y_{\sigma'}$ and hence~$[D_\rho]\in~\Gamma(Y_{\sigma'})$, but $\Gamma(Y_{\sigma'})=\Gamma(Y_\sigma)$ by Corollary~\ref{cor:necessary}. Therefore, we need to prove the sufficiency. This is the content of Lemma~\ref{lem:transition} below.

First, we show that condition~\eqref{eq:deletion} can be checked directly on the colored fan of~$X$. Recall the standard exact sequence
\begin{equation}\label{eq:pic}
 0\longrightarrow M
 \xrightarrow{\ \psi }
 \bigoplus_{D\in\cD^B}\Z D
 \longrightarrow\Pic(X)\longrightarrow0,
\end{equation}
where
\begin{equation}\label{eq:principal}
 \psi(m)=\divisor(f_m)=
 \sum_{\rho\in\cF(1)}\pair{v_\rho}{m}D_\rho
 +\sum_{F\in\cD}\pair{v_F}{m}F.
\end{equation}

\begin{proposition}\label{prop:inequalities}
Let $\sigma\in\cF$ and let $\rho\in\sigma(1)$. Condition~\eqref{eq:deletion} holds if and only if there is an element $m\in M$ such that
\begin{align}
 \pair{v_\rho}{m}&=-1,\label{eq:root1}\\
 \pair{v_{\rho'}}{m}&=0
 &&\text{for }\rho'\in\sigma(1)\setminus\{\rho\},\label{eq:root2}\\
 \pair{v_{\eta}}{m}&\geq0
 &&\text{for }\eta\in\cF(1)\setminus\sigma(1),\label{eq:root3}\\
 \pair{v_F}{m}&\geq0
 &&\text{for }F\in\cD.\label{eq:root4}
\end{align}
\end{proposition}

\begin{proof}
If $[D_\rho]\in\Gamma(Y_\sigma)$, then, by~\eqref{eq:monoid-fan}, there is an effective divisor $E\sim D_\rho$ whose prime components are among the colors $F\in\cD$ and the divisors $D_\eta$, $\eta\in\cF(1)\setminus\sigma(1)$. By the exactness of~\eqref{eq:pic} at the middle term,
\[
 E-D_\rho=\divisor(f_m)
\]
for some $m\in M$. Comparing coefficients gives conditions~\eqref{eq:root1}--\eqref{eq:root4}.

Conversely, these conditions imply that $E=D_\rho+\divisor(f_m)$ is effective and its prime components are among the $B$-stable prime divisors not containing $Y_\sigma$. This yields\[[D_\rho]=[E]\in\Gamma(Y_\sigma),\] and finishes the proof.
\end{proof}

\subsection{Global vector fields and \texorpdfstring{$B$}{B}-root subgroups}\label{subsec:roots}

A one-parameter unipotent subgroup~$U\subseteq\Aut(X)$ normalized by the group $B$ is called a~$B$\emph{-root subgroup}. Its weight is the character by which $B$ acts on $\Lie U$ under the adjoint representation. Properties of $B$-root subgroups on spherical varieties were studied in depth in~\cite{AA}, \cite{AZ22}, and~\cite{AZ23}.

Denote by $\operatorname{GB}(X)\subseteq\Aut^0(X)$ the subgroup generated by the image of $G$ and all $B$-root subgroups. In this subsection, we show that if condition~\eqref{eq:deletion} holds, then orbits~$Y_\sigma$ and~$Y_{\sigma'}$ can be connected by a $B$-root subgroup. The following description of global vector fields will help us construct such a subgroup.

\begin{theorem}\cite[Proposition 4.1.1]{BB}\label{thm:BB}
Let $X$ be a smooth complete toroidal spherical variety and let $D_1,\dotsc,D_k$ be all the $G$-stable prime divisors on $X$. There is an exact sequence of $G$-modules
\begin{equation}\label{eq:BB}
 0\longrightarrow H^0(X,\cS_X)
 \longrightarrow H^0(X,\cT_X)
 \xrightarrow{\nu}
 \bigoplus_{i=1}^k H^0(D_i,\cO_{D_i}(D_i))
 \longrightarrow0,
\end{equation} where $\cT_X$ is the tangent sheaf of $X$ and $\cS_X$ is the subsheaf of $\cT_X$ consisting of vector fields tangent to every $D_i$. The map $\nu$ takes the normal components along the divisors~$D_1,\dotsc,D_k$.
\end{theorem}

Bien and Brion proved this statement in~\cite{BB} for the class of varieties which are called \emph{regular}. But in the same paper they also proved that a smooth complete $G$-variety is regular if and only if it is spherical and toroidal; see~\cite[Proposition 2.2.1]{BB}.

Recall also that $H^0(X,\cT_X)=\Lie\Aut^0(X)$ since $X$ is complete. The next lemma plays a central role in the proof of the main theorem.

\begin{lemma}[Ray-removal lemma]\label{lem:transition}
Let $\sigma\in\cF$ and $\rho\in\sigma(1)$, and let $\sigma'$ be the face of $\sigma$ generated by $\sigma(1)\setminus\{\rho\}$. Suppose that condition~\eqref{eq:deletion} holds, and let $m\in M$ satisfy~\eqref{eq:root1}--\eqref{eq:root4}. Such $m$ exists by Proposition~\ref{prop:inequalities}. Then there is a $B$-root subgroup $U_m\subseteq\Aut^0(X)$ of weight $m$ such that the orbit $U_mx$ meets $Y_{\sigma'}$ for some point $x\in Y_\sigma$.
In particular, $Y_\sigma$ and $Y_{\sigma'}$ lie in the same $\operatorname{GB}(X)$-orbit and $\Gamma(Y_\sigma)=\Gamma(Y_{\sigma'})$.
\end{lemma}

\begin{proof}
\emph{Step 1: a $B$-semiinvariant section of $\cO_X(D_\rho)$.} Since $D_\rho$ is $G$-stable, the line bundle~$\cO_X(D_\rho)$ has a canonical $G$-linearization as a subsheaf of $\K(X)$. Its canonical section $s_\rho$, represented by $1$, is $G$-invariant. Consequently
\[
 s=f_m s_\rho
\]
is a rational $B$-semiinvariant section of $\cO_X(D_\rho)$ of weight $m$, and $m\ne0$ since $\pair{v_\rho}{m}=-1$. By~\eqref{eq:principal} and conditions~\eqref{eq:root1}--\eqref{eq:root4}, the divisor of $s$ equals
\[
 E=D_\rho+\divisor(f_m)=\sum_{\eta\in\cF(1)\setminus\sigma(1)}\pair{v_\eta}{m}D_\eta+\sum_{F\in\cD}\pair{v_F}{m}F.
\]
It is effective, so $s\in H^0(X,\cO_X(D_\rho))$. The colors do not contain $Y_\sigma$ since $X$ is toroidal, and the divisors $D_\eta$ with $\eta\in\cF(1)\setminus\sigma(1)$ also do not contain $Y_\sigma$. Therefore we can choose a point $x\in Y_\sigma\setminus\supp E$. We have $s(x)\ne0$ and the restriction
\[
 \overline s=s|_{D_\rho}\in H^0(D_\rho,\cO_{D_\rho}(D_\rho))
\]
is nonzero at $x$ since $Y_\sigma\subseteq D_\rho$.

\emph{Step 2: a $B$-eigenvector in $H^0(X,\cT_X)$.} Consider the vector $\tilde s$ in the right-hand side of~\eqref{eq:BB} whose $\rho$-component is $\overline s$ and whose other components are zero. It is a $B$-eigenvector of weight $m$, since the restriction to the $G$-stable divisor $D_\rho$ is $G$-equivariant. Since $G$ is reductive, the surjection $\nu$ has a $G$-equivariant splitting. Therefore, there is a $B$-eigenvector
\[
 \xi_m\in H^0(X,\cT_X)=\Lie\Aut^0(X)
\]
of weight $m$, such that $\nu(\xi_m)=\tilde s$. The normal components of $\xi_m$ along all $D_{\rho'}$ with $\rho'\ne\rho$ are zero, so $\xi_m$ is tangent to these divisors. The normal component of $\xi_m$ along $D_\rho$ is $\overline s$, and $\overline s(x)\ne0$ which yields $\xi_m(x)\notin T_xD_\rho$.

\emph{Step 3: the $B$-root subgroup.} Since $\xi_m$ is a $B$-eigenvector, we have $\operatorname{Ad}(t)\xi_m=m(t)\xi_m$ for~$t\in T$, and $m(T)=\K^\times$ since $m\ne0$. This means that the finite set of eigenvalues of~$\xi_m$ in a faithful representation of $\Aut^0(X)$ is invariant under multiplication by~$\K^\times$, and hence all its eigenvalues are zero. Thus, $\xi_m$ is nilpotent and the group
\[
 U_m = \{\exp(t\xi_m) \mid t\in\K\}\subseteq\Aut^0(X)
\]
is a $B$-root subgroup.

\emph{Step 4: the transition.} Consider the following group:
\[
 K = \Big(\bigcap_{\rho'\in\cF(1)\setminus\{\rho\}} \operatorname{Stab}_{\Aut^0(X)}(D_{\rho'})\Big)^0\subseteq \Aut^0(X).
\]
Its Lie algebra is
\[
 \Lie K = \{\zeta \in H^0(X, \cT_X) \mid \text{$\zeta$ is tangent to $D_{\rho'}$ for $\rho'\ne\rho$}\}.
\]
By Step~2 we have $\xi_m\in\Lie K$, hence $U_m\subseteq K$.

The orbit $U_mx$ is not contained in $D_\rho$. Otherwise, $\xi_m$ would be tangent to $D_\rho$ which is a contradiction with Step 2. Choose $y\in U_mx\setminus D_\rho$. For every $\rho'\ne\rho$ we have
\[
 y \in D_{\rho'} \quad \Longleftrightarrow \quad  x\in D_{\rho'} \quad\Longleftrightarrow\quad \rho'\in\sigma(1),
\]
since $U_m\subseteq K$. Hence $y$ lies in $D_{\rho'}$ if and only if $\rho'\in\sigma(1)\setminus\{\rho\}=\sigma'(1)$. It remains to say that a $G$-orbit is uniquely determined by the set of $G$-stable prime divisors containing it, so $y\in Y_{\sigma'}$.

Finally, $Y_\sigma$ and $Y_{\sigma'}$ lie in the same $\operatorname{GB}(X)$-orbit, since $\operatorname{GB}(X)$ contains $G$. The equality~$\Gamma(Y_\sigma)=\Gamma(Y_{\sigma'})$ follows from Corollary~\ref{cor:necessary} since $\operatorname{GB}(X)\subseteq\Aut^0(X)$.
\end{proof}

\subsection{The orbit criterion}\label{subsec:criterion}

Now we are ready to prove the main result of this section.

\begin{theorem}\label{thm:main-complete}
    Let $X$ be a smooth complete toroidal spherical $G$-variety, and let $x, x'\in X$. Then the following conditions are equivalent:
    \begin{enumerate}
        \item $x$ and $x'$ lie in the same $\Aut^0(X)$-orbit;
        \item $x$ and $x'$ lie in the same $\operatorname{GB}(X)$-orbit;
        \item $\Gamma(Gx)=\Gamma(Gx')$.
    \end{enumerate}
\end{theorem}

\begin{proof}
The implication $(2)\Rightarrow(1)$ is obvious since $\operatorname{GB}(X)\subseteq\Aut^0(X)$. The implication $(1)\Rightarrow(3)$ is Corollary~\ref{cor:necessary}. We need to prove $(3)\Rightarrow(2)$.

Since $\operatorname{GB}(X)$ contains the image of $G$, it suffices to show that the $G$-orbits $Gx$ and $Gx'$ lie in the same $\operatorname{GB}(X)$-orbit. We have $Gx=Y_\sigma$ and $Gx'=Y_\tau$ for some cones $\sigma,\tau\in \cF$, and $\Gamma(Y_{\sigma})=\Gamma(Y_{\tau})$ by assumption.

Take $\rho\in\sigma(1)\setminus\tau(1)$. The divisor $D_\rho$ does not contain $Y_\tau$ and hence
\[
 [D_\rho]\in\Gamma(Y_\tau)=\Gamma(Y_\sigma).
\]
By Lemma~\ref{lem:transition}, the orbit $Y_\sigma$ and the orbit $Y_{\sigma'}$, where $\sigma' = \operatorname{cone}(\sigma(1)\setminus \rho)$, lie in the same $\operatorname{GB}(X)$-orbit and $\Gamma(Y_{\sigma'})=\Gamma(Y_\sigma)=\Gamma(Y_\tau)$. Hence we can repeat this process for all rays in $\sigma(1)\setminus\tau(1)$. In this way we connect $Y_\sigma$ with $Y_{\sigma\cap\tau}$. Indeed, $\sigma\cap\tau$ is a common face of $\sigma$ and $\tau$, and its rays are exactly their common rays~$\sigma(1)\cap\tau(1)$. We use the fact that all cones are regular since $X$ is smooth. Applying the same argument to $\tau$, we connect $Y_\tau$ with $Y_{\sigma\cap\tau}$. Therefore $Y_\sigma$ and $Y_\tau$ lie in the same $\operatorname{GB}(X)$-orbit.
\end{proof}

The following corollary has the same proof as Corollary~\ref{cor:full-affine}.

\begin{corollary}\label{cor:full-complete}
Let $X$ be a smooth complete toroidal spherical variety. Two $G$-orbits $Y,Y'$ are contained in the same $\Aut(X)$-orbit if and only if there is an automorphism $\varphi\in\Aut(X)$ such that
\[
 \overline{\varphi}(\Gamma(Y))=\Gamma(Y'),
\]
where $\overline{\varphi}$ denotes the image of $\varphi$ in $\Aut(\Pic(X))$.
\end{corollary}

In the complete case, the image of $\Aut(X)$ in $\Aut(\Pic(X))$ is a finite group. Indeed, the effective cone $\operatorname{Eff}(X)\subset \Pic(X)$ is generated by the classes of $B$-stable prime divisors. It follows from the fact that $B$ has a fixed point on $|D|$ for any divisor $D$ by Borel's theorem since $|D|$ is complete. The cone $\operatorname{Eff}(X)$ is also pointed because an effective divisor on a complete variety cannot be principal. This means that the image of~$\Aut(X)$ in~$\Aut(\Pic(X))$ permutes indecomposable vectors generating the cone $\operatorname{Eff}(X)$, and hence it should be finite. Therefore, this reduces the problem of describing the orbits to the study of the action of a finite group on $\Pic(X)$.

Proposition~\ref{prop:inequalities} allows to describe the $\Aut^0(X)$-orbits directly from the colored fan of~$X$. Consider the graph $\mathcal G_\cF$ whose vertices are the cones $\sigma\in\cF$, and two cones~$\sigma$ and~$\sigma'=\operatorname{cone}(\sigma(1)\setminus\{\rho\})$, $\rho\in\sigma(1)$, are joined by an edge if and only if there is an element~$m\in M$ satisfying conditions~\eqref{eq:root1}--\eqref{eq:root4}.

\begin{corollary}\label{cor:graph}
The $\Aut^0(X)$-orbits in $X$ correspond to the connected components of $\mathcal{G}_\cF$, i.e. two $G$-orbits $Y_\sigma$ and $Y_\tau$ lie in the same $\Aut^0(X)$-orbit if and only if $\sigma$ and $\tau$ lie in the same connected component of $\mathcal G_\cF$. If two cones~$\sigma$ and~$\tau$ are joined by an edge, then the corresponding orbits~$Y_{\sigma}$ and~$Y_\tau$ can be connected by a $B$-root subgroup.
\end{corollary}

\begin{proof}
If $\sigma$ and $\sigma'$ are joined by an edge, then $Y_\sigma$ and $Y_{\sigma'}$ are connected by a $B$-root subgroup by Lemma~\ref{lem:transition}. In particular, they lie in the same $\Aut^0(X)$-orbit. Conversely, if $Y_\sigma$ and $Y_\tau$ lie in the same $\Aut^0(X)$-orbit, then $\Gamma(Y_\sigma)=\Gamma(Y_\tau)$ by Theorem~\ref{thm:main-complete}. The proof of the implication $(3)\Rightarrow(2)$ in Theorem~\ref{cor:full-complete} gives the path in $\mathcal G_\cF$ from $\sigma$ and from $\tau$ to $\sigma\cap\tau$.
\end{proof}

\begin{example}
    Consider the \emph{space of complete conics}
\[
 X=\{([q],[q'])\in\PP(\operatorname{S}^2\K^3)\times\PP(\operatorname{S}^2\K^{3*})\mid q\cdot q'\in\K\cdot I\},
\]
where $I$ is the identity matrix. This variety is described in detail in~\cite[Example 17.3]{Tim11}.
    
    It is a smooth complete toroidal spherical variety for $G=\SL_3$ naturally acting on both factors. The open $G$-orbit is $\SL_3/\SO_3$, and its weight lattice is $M=\gen{2\alpha_1,2\alpha_2}$, where $\alpha_1,\alpha_2$ are the simple roots of $\SL_3$. The variety $X$ is also simple, and hence its fan $\mathcal F$ consists of one cone coinciding with $\mathcal V$. It has two $G$-stable prime divisors~$E_1$ and~$E_2$ and two colors~$D_1$ and~$D_2$. The divisor~$E_1$ corresponds to pairs of lines while the divisor~$E_2$ corresponds to double lines. The color data in the lattice $N$ with the basis~$n_1=-\omega_1^\vee/2$,~$n_2=-\omega_2^\vee/2$, where~$\omega_1^\vee$ and~$\omega_2^\vee$ are the fundamental coweights of~$\SL_3$, are shown in Figure~\ref{fig:conics}.

\begin{figure}[h!]
    \centering
    \begin{tikzpicture}[scale=0.85, >=Stealth, font=\footnotesize]
        
        \fill[blue!10] (0,0) -- (3.5,0) -- (3.5,3.5) -- (0,3.5) -- cycle;
        
        \foreach \x in {-3,-2,-1,0,1,2,3} {
            \foreach \y in {-3,-2,-1,0,1,2,3} {
                \fill[gray!50] (\x,\y) circle (1pt);
            }
        }
        
        \draw[->, thin, gray!80] (-3.5,0) -- (4,0) node[right, text=black] {$n_1$};
        \draw[->, thin, gray!80] (0,-3.5) -- (0,4) node[above, text=black] {$n_2$};
        
        \fill (0,0) circle (1.5pt) node[below left] {$0$};
        
        \node at (2, 2) [blue!80!black] {$\mathcal{V} = \mathcal F$};
        
        \draw[->, very thick, blue!80!black] (0,0) -- (1,0) node[below right=-1pt] {$v_{E_1} = (1,0)$};
        \draw[->, very thick, blue!80!black] (0,0) -- (0,1) node[above right=-3pt] {$v_{E_2} = (0,1)$};
        
        \draw[->, very thick, red] (0,0) -- (-2,1) node[above left] {$v_{D_1} = (-2,1)$};
        \draw[->, very thick, red] (0,0) -- (1,-2) node[below right] {$v_{D_2} = (1,-2)$};
        
        \fill[red] (-2,1) circle (1.5pt);
        \fill[red] (1,-2) circle (1.5pt);

    \end{tikzpicture}
    \caption{The colored fan for the space of complete conics.}\label{fig:conics}
\end{figure}
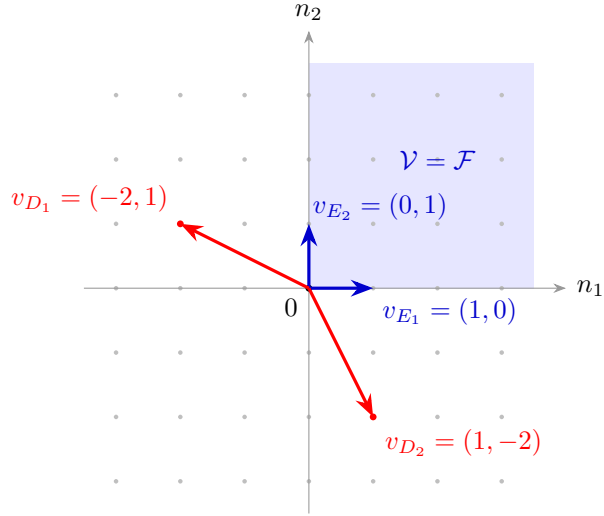

Thus we have 4 cones: $\sigma_{12}=\operatorname{cone}(v_{E_1}, v_{E_2})$, $\sigma_1=\operatorname{cone}(v_{E_1})$, $\sigma_2=\operatorname{cone}(v_{E_2})$, $\sigma_0=\{0\}.$ One can easily see that for every cone and any of its rays the system~\eqref{eq:root1}-\eqref{eq:root4} is inconsistent. For this write $m=(a,b)$ in the basis of $M$ dual to $(n_1,n_2)$. Then we have
\begin{itemize}
\item for $\sigma_{12}$ and the ray of $E_1$, conditions~\eqref{eq:root1} and~\eqref{eq:root2} give $a=-1$ and $b=0$, so $\pair{v_{D_2}}{m}=-1$;
\item for $\sigma_{12}$ and the ray of $E_2$, we get $a=0$ and $b=-1$, so $\pair{v_{D_1}}{m}=-1$;
\item for $\sigma_1$, conditions~\eqref{eq:root1} and~\eqref{eq:root3} give $a=-1$ and $b\geq0$, so $\pair{v_{D_2}}{m}=-1-2b<0$;
\item for $\sigma_2$, we get $b=-1$ and $a\geq0$, so $\pair{v_{D_1}}{m}=-2a-1<0$;
\item the cone $\sigma_0$ has no rays.
\end{itemize}
These calculations show that in each case condition~\eqref{eq:root4} fails. Hence the graph $\mathcal G_\cF$ has no edges, and by Corollary~\ref{cor:graph} the $G$-orbits coincide with the $\Aut^0(X)$-orbits for the space of complete conics.
\end{example}

\subsection*{Data Availability Statement}
Data sharing is not applicable to this article as no datasets were generated or analysed during the current study.

\subsection*{Conflict of Interest}
The authors declare no conflict of interest.

\end{document}